\documentclass{amsart}
\usepackage{url,subcaption,amsmath,amssymb,amsthm,booktabs,algpseudocode,algorithm,graphicx,lscape}
\usepackage[onehalfspacing]{setspace}
\newtheorem{thm}{Theorem}
\newtheorem{lem}{Lemma}
\newtheorem{prop}{Proposition}
\newtheorem{assum}{Assumption}
\newtheorem{prob}{Problem}
\newcommand{\nat}{\mathbb{N}}
\newcommand{\real}{\mathbb{R}}
\newcommand{\ip}[1]{\left\langle{}#1\right\rangle}
\newcommand{\norm}[1]{\left\lVert{}#1\right\rVert}
\DeclareMathOperator*{\argmin}{argmin}
\newif\ifarxiv\arxivtrue

\begin{document}
\title[Subgradient Algorithms with the Line Search]{Incremental and Parallel Machine Learning Algorithms with Automated Learning Rate Adjustments}
\author[K.~Hishinuma]{Kazuhiro Hishinuma}
\author[H.~Iiduka]{Hideaki Iiduka}
\begin{abstract}
The existing machine learning algorithms for minimizing the convex function over a closed convex set suffer from slow convergence because their learning rates must be determined before running them.
This paper proposes two machine learning algorithms incorporating the line search method, which automatically and algorithmically finds appropriate learning rates at run-time.
One algorithm is based on the incremental subgradient algorithm, which sequentially and cyclically uses each of the parts of the objective function; the other is based on the parallel subgradient algorithm, which uses parts independently in parallel.
These algorithms can be applied to constrained nonsmooth convex optimization problems appearing in tasks of learning support vector machines without adjusting the learning rates precisely.
The proposed line search method can determine learning rates to satisfy weaker conditions than the ones used in the existing machine learning algorithms.
This implies that the two algorithms are generalizations of the existing incremental and parallel subgradient algorithms for solving constrained nonsmooth convex optimization problems.
We show that they generate sequences that converge to a solution of the constrained nonsmooth convex optimization problem under certain conditions.
The main contribution of this paper is the provision of three kinds of experiment showing that the two algorithms can solve concrete experimental problems faster than the existing algorithms.
First, we show that the proposed algorithms have performance advantages over the existing ones in solving a test problem.
Second, we compare the proposed algorithms with a different algorithm Pegasos, which is designed to learn with a support vector machine efficiently, in terms of prediction accuracy, value of the objective function, and computational time.
Finally, we use one of our algorithms to train a multilayer neural network and discuss its applicability to deep learning.

\end{abstract}
\maketitle

\section{Introduction}
In this paper, we consider a technique to adjust the learning rates that appear in subgradient algorithms for letting a generated sequence converge to an optimal solution.
The subgradient algorithm \cite[Section~8.2]{bertsekas} and its variants \cite{psm,ism,pegasos} have been proposed as ways of solving the problem of minimizing a nonsmooth, convex function over a closed convex set by iterative processes like the steepest descent method for dealing with a smooth, convex function.
These methods iterate the current approximate solution by shifting it along a descent direction at that point by a given degree called a \emph{learning rate}.
Although descent directions are decided on the basis of the subgradient at each point, the learning rates are generally decided for theoretical reasons for ensuring the convergence of the generated sequence.
This implies that subgradient algorithms can be run more efficiently if we can choose more suitable learning rates concerning the objective function at each iteration.
Therefore, we should consider how to choose better learning rates while at the same time maintaining the convergence properties.

We can reduce a lot of practical problems to ones solvable with subgradient algorithms, that is, problems of minimizing a nonsmooth, convex function over a closed convex set.
One of the important applications is learning with a support vector machine.
Support vector machines are effective and popular classification learning tools \cite{ml1,ml2,ml3,pegasos}.
The task of learning with a support vector machine is cast as an empirical loss minimization with a penalty term for the norm of the classifier that is being learned \cite[Problem~(1)]{pegasos}.
If this loss objective function is convex, we can handle this learning task by minimizing a nonsmooth, convex function over a closed convex set.
There are practical optimization algorithms for solving this minimization problem, such as Pegasos \cite{pegasos}, the incremental subgradient algorithm \cite{ism}, and the parallel subgradient algorithm \cite{psm}.
These algorithms are variants of the subgradient algorithm.
They iteratively choose training examples and improve their approximation by using a part of the objective function which corresponds to the chosen examples.
Not limited to machine learning, there exist many applications of minimizing a nonsmooth, convex function over a closed convex set, such as signal recovery \cite{combettes2003}, bandwidth allocation \cite{iiduka2013}, and beamforming \cite{yamada2007}.
Hence, making the performance of these algorithms better would increase the efficiency of these applications.
Here, we attempt to do so by modifying the selection of the learning rate.

Pegasos is a stochastic subgradient algorithm with a carefully chosen learning rate that is designed for efficiently learning with a support vector machine \cite{pegasos}.
This learning rate is determined from the regularization constant of the penalty term.
Hence, this algorithm can improve approximate solutions without having to adjust their learning rates for each individual learning task.
However, it is specialized to learning with a support vector machine, and it cannot be applied to other applications such as deep learning.

The sequential minimal optimization (SMO) algorithm \cite{platt1998} is also used for learning with a support vector machine.
This algorithm can be applied to a quadratic programming optimization problem appearing in learning with a dual form of a support vector machine and can solve it with a small amount of memory and quickly \cite{svm-book,platt1998}.
However, this algorithm deals with the dual form of the optimization problem; as such, the number of objective variables is likely to be large when many instances are given to the learning task.
Furthermore, the class of problem that this algorithm can deal with is limited to quadratic programming.
This implies that it cannot be applied to general nonsmooth, convex programming.

In the field of mathematical optimization, the incremental and parallel subgradient algorithms \cite{psm,ism} are useful for solving problems involving the minimization of a nonsmooth, convex function over a closed convex set.
The incremental subgradient algorithm \cite{ism} minimizes the objective function by using alternately one of the functions composing the summed objective, while the parallel subgradient algorithm \cite{psm} minimizes it by using all of the composing functions independently.
Since the parallel subgradient algorithm treats each of the composing functions independently, computations with respect to each function can be parallelized.
It is expected that parallelization shortens the computational time of learning.
This implies that the parallel subgradient algorithm can learn support vector machines for larger datasets and/or in a shorter time compared with other algorithms.

A weak point of the incremental and parallel subgradient algorithms \cite{psm,ism} is that they need to have suitably adjusted learning rates in order to run efficiently.
However, the suitable learning rate depends on various factors, such as the number of the composing objective functions, number of dimensions, the shape of each objective function and constraint set, and the selection of subgradients.
This implies that it is too difficult to choose a suitable learning rate before run-time.
In contrast, Pegasos \cite{pegasos} uses a concrete learning rate optimized for the task of learning with a support vector machine and does not require this learning rate to be adjusted.
Therefore, it can be used more easily than the incremental and parallel subgradient algorithms \cite{psm,ism}.

In unconstrained minimization algorithms, line searches are used to select a suitable learning rate \cite{linesearch1,linesearch2}.
In particular, the \emph{Wolfe conditions} \cite{wolfe} are learning rate criteria for the line search.
The Wolfe conditions are such that the learning rate must satisfy a sufficient decrease condition and a curvature condition \cite[Chapter 3]{nocedal}.
The sufficient decrease condition is that the learning rate is acceptable only if its function value is below a linear function with a negative slope.
This condition ensures that the algorithms update an approximation to a better one.
However, it is not enough to ensure that the algorithm makes reasonable progress because it will do so for all sufficiently small learning rates.
Therefore, a curvature condition is invoked that generates a sequence further enough along the chosen direction.

Motivated by the idea of the line search, this paper proposes novel incremental and parallel subgradient algorithms that can run efficiently without precise learning rate adjustments.
Reference~\cite{pgmls} describes a gradient-projection algorithm with a line search that minimizes the objective function.
However, this algorithm assumes that the objective function is differentiable.
In addition, it is designed for single-core computing; it is not useful in multi-core computing.
Reference~\cite{beltran2005} proposes the radar subgradient algorithm, which is a variant of the subgradient algorithm including a procedure for finding an effective learning rate by using a line search at each iteration.
The line search method used in Reference~\cite{beltran2005} is inspired by the cutting-plane method and works out a learning rate with the first-order information.
However, this algorithm deals with the whole objective function and cannot use a part of the objective function at each iteration.
This implies that it cannot be used in applications that give information to the algorithm through a data stream.
In addition, the line search method used in Reference~\cite{beltran2005} may fail and is distinct from the line search proposed in this paper.
Hence, combining this line search method with the one we propose may have a complementary effect when the properties of the optimization problem are disadvantageous to one of the algorithms.
Reference~\cite{iidukacg} gives an algorithm for solving fixed point problems, covering the constrained minimization problem discussed in this paper, with a line search.
This algorithm has a fast convergence property, though it decides only the coefficient of the convex combination and is not designed for multi-core computing.
The algorithm in \cite{hayashiy,psm,psm2,star,ism} requires a suitable learning rate in order to converge efficiently.
However, as we mentioned before, the learning rate is very difficult to adjust.

In contrast to previous reports, this paper proposes incremental and parallel subgradient algorithms with a line search to find better learning rates than the ones used in the existing algorithms.
To realize this proposal, we extend the concept of the \emph{learning rate} to a \emph{step-range}, which is a set of candidates for the learning rate.
The line search procedure is given a step-range and chooses the most suitable learning rate among it at run-time.
Using a line search with a step-range has three merits.
First, the suitable learning rates chosen by the line search accelerate the algorithms and make their solutions better.
Section~\ref{sec:experi} shows that the proposed algorithms gave better solutions than the one given by Pegasos \cite{pegasos} when they all ran the same number of iterations.
The second merit is that we do not need to adjust the learning rate precisely.
The existing incremental and parallel subgradient algorithms \cite{psm,ism} cannot converge efficiently without appropriate adjustments to their learning rates.
This is their weak point in comparison with Pegasos \cite{pegasos}.
In contrast, the proposed algorithms only need step-ranges, i.e. rough candidates, to converge efficiently, because the line search automatically chooses the learning rates from among the step-range.
Hence, they can be easily used to learn support vector machines.
Finally, the proposed algorithms can be applied to difficult problems whose suitable learning rates cannot be chosen beforehand.
Section~\ref{sec:analys} provides a condition on the step-range compositions to ensure they converge to an optimizer of the problem.
Hence, even if a suitable learning rate cannot be specified beforehand, the line search can algorithmically find one at run-time and make the algorithms converge efficiently to an optimizer.
We show that our algorithms converge to an optimizer to the problem when the step-range is diminishing.
In addition, if the step-range is a singleton set, they coincide with the existing incremental and parallel subgradient algorithms \cite{psm,ism}.
Hence, the step-range is a generalization of the learning rates used in the existing algorithms.

We compared the proposed algorithms with Pegasos \cite{pegasos} and the SMO algorithm on various datasets \cite{uci,mnist,libsvm} for binary and multiclass classification.
The results of the comparison demonstrated that the proposed algorithms perform better than the existing ones in terms of the value of the objective function for learning with a support vector machine and in terms of computational time.
In particular, the parallel subgradient algorithm dramatically reduced the computational times of the learning tasks.

Stochastic subgradient algorithms are useful for learning with a multilayer neural network habitually \cite{bottou1991}.
The incremental subgradient algorithm is a specialization of the stochastic subgradient algorithm.
Therefore, we can use one of our algorithms, a variant of the incremental subgradient algorithm, to train a multilayer neural network.
We compared it with two other variants of the incremental subgradient algorithm.
The results show that our algorithm can minimize the objective function of the trained neural network more than the others.
This ability implies that it is also useful for training not only SVMs but also neural networks, including ones for deep learning.

This paper is organized as follows.
Section~\ref{sec:prelim} gives the mathematical preliminaries and mathematical formulation of the main problem.
Section~\ref{sec:method} presents our algorithms.
We also show the fundamental properties of these algorithms that are used to prove the main theorems.
Section~\ref{sec:analys} presents convergence analyses.
Section~\ref{sec:experi} describes numerical comparisons of the proposed algorithms with the existing ones in Reference~\cite{psm,ism,pegasos} using concrete machine learning datasets \cite{uci,mnist,libsvm}.
In this section, we also describe how to use one of the proposed algorithms to train a multilayer neural network for recognizing handwritten digits.
Section~\ref{sec:conclu} concludes this paper.

\section{Mathematical Preliminaries}\label{sec:prelim}
Let $\real^N$ be an $N$-dimensional Euclidean space with the standard Euclidean inner product $\langle\cdot,\cdot\rangle:\real^N\times\real^N\to\real$ and its induced norm defined by $\|x\|:=\langle{}x,x\rangle^\frac{1}{2}$.
We define the notation $\nat:=\{1,2,\ldots\}$ as the set of all natural numbers.
Let $x_n\to{}x$ denote that the sequence $\{x_n\}\subset\real^N$ converges to a point $x\in\real^N$.

A \emph{subgradient} $g$ of a convex function $f:\real^N\to\real$ at a point $x\in\real^N$ is defined by $g\in\real^N$ such that $f(x)+\langle{}y-x,g\rangle{}\le{}f(y)$ for all $y\in\real^N$.
The set of all subgradients at a point $x\in\real^N$ is denoted as $\partial{}f(x)$ \cite{rockafellar}, \cite[Section 7.3]{noco}.

The metric projection onto a nonempty, closed convex set $C\subset\real^N$ is denoted by $P_C:\real^N\to{}C$ and defined by $\|x-P_C(x)\|=\inf_{y\in{}C}\|x-y\|$ \cite[Section~4.2, Chapter~28]{bc}.
$P_C$ satisfies the nonexpansivity condition \cite[Subchapter~5.2]{noco}; i.e. $\|P_C(x)-P_C(y)\|\le\|x-y\|$ for all $x,y\in\real^N$.

\subsection{Main Problem}
We will consider the following optimization problem \cite{psm,ism}: let $f_i:\real^N\to[0,\infty)\quad(i=1,2,\ldots,K)$ be convex, continuous functions and let $C$ be a nonempty, closed convex subset of $\real^N$.
Then,
\begin{align}\begin{split}
\text{minimize }&f(x):=\sum_{i=1}^Kf_i(x),\\
\text{Subject to }&x\in{}C.\label{eq:main}
\end{split}\end{align}

Let us discuss Problem~\eqref{eq:main} in the situation that a closed convex subset $C$ of an $N$-dimensional Euclidean space $\real^N$ is simple in the sense that $P_C$ can be computed within a finite number of arithmetic operations.
Examples of a simple, closed convex set $C$ are a closed ball, a half-space, and the intersection of two half-spaces \cite[Examples 3.16 and 3.21, and Proposition 28.19]{bc}.

The task of learning with a support vector machine can be cast as Problem (1) \cite[Problem~(1)]{pegasos}.
Furthermore, there are a lot of applications not limited to learning with a support vector machine when $f$ is nonsmooth but convex on $\real^N$ and when $C\subset\real^N$ is simple.
For example, minimizing the total variation of a signal over a convex set and Tykhonov-like problems with $L^1$-norms \cite[I.~Introduction]{dratch} are able to be handled as Problem~\eqref{eq:main}.
Application of Problem~\eqref{eq:main} to learning with a support vector machine will be described in Section~\ref{sec:experi}.

The following assumptions are made throughout this paper.
\begin{assum}[Subgradient Boundedness{\cite[Assumption~2.1]{ism}}]\label{assum:sgbdd}
For all $i=1,2,\ldots,K$, there exists $M_i\in(0,\infty)$ such that
\begin{align*}\|g\|\le{}M_i\quad(x\in{}C; g\in\partial{}f_i(x)).
\end{align*}
We define a constant $M:=\sum_{i=1}^KM_i$.
\end{assum}
\begin{assum}[Existence of Optimal Solution{\cite[Proposition~2.4]{ism}}]\label{assum:exsol}
$\argmin_{x\in{}C}f(x)\neq\emptyset$.
\end{assum}

\section{Proposed Algorithms and their Fundamental Properties}\label{sec:method}
\subsection{Incremental Subgradient Algorithm}
This subsection presents the incremental subgradient algorithm, Algorithm~\ref{alg:ngism}, for solving Problem~\eqref{eq:main}.
\begin{algorithm}
  \caption{Incremental Subgradient Algorithm}
  \label{alg:ngism}
  \begin{algorithmic}[1]
    \Require{$\forall{}n\in\nat,[\underline{\lambda}_n,\overline{\lambda}_n]\subset(0,\infty)$.}
    \State{$n\gets{}1$, $x_1\in{}C$.}
    \Loop
      \State{$y_{n,0}:=x_n$.}
      \For{$i=1,2,\ldots,K$}\Comment{In sequence}
        \State{$g_{n,i}\in\partial{}f_i(y_{n,i-1})$.}\label{all:ngism:start}
        \State{$\lambda_{n,i}\in[\underline{\lambda}_n,\overline{\lambda}_n]$.}\label{all:ngism:step}\Comment{By a line search algorithm}
        \State{$y_{n,i}:=P_C(y_{n,i-1}-\lambda_{n,i}g_{n,i})$.}\label{all:ngism:end}
      \EndFor
      \State{$x_{n+1}:=y_{n,K}$.}
      \State{$n\gets{}n+1$.}
    \EndLoop
  \end{algorithmic}
\end{algorithm}
Let us compare Algorithm~\ref{alg:ngism} with the existing one\cite{ism}.
The difference is Step~\ref{all:ngism:step} of Algorithm~\ref{alg:ngism}.
The learning rate $\lambda_n$ of the existing algorithm must be decided before the algorithm runs.
However, Algorithm~\ref{alg:ngism} only needs the step-range $[\underline{\lambda}_n,\overline{\lambda}_n]$.
A learning rate within the range used by Algorithm~\ref{alg:ngism} can be automatically determined at run-time.
Algorithm~\ref{alg:ngism} coincides with the incremental subgradient algorithm when the given step-range $[\underline{\lambda}_n,\overline{\lambda}_n]$ is a singleton set, i.e. $\underline{\lambda}_n:=\overline{\lambda}_n:=\lambda_n$, which means that it is a generalization of the algorithm in \cite{ism}.
In this case, Algorithm~\ref{alg:ngism} chooses only one learning rate $\lambda_n$ from the singleton step-range $[\underline{\lambda}_n,\overline{\lambda}_n]=\{\lambda_n\}$.

This difference has three merits.
First, the suitably chosen learning rates in Step~\ref{all:ngism:step} accelerate convergence and make the solutions more accurate.
Second, Algorithm~\ref{alg:ngism} does not require the learning rate to be precisely adjusted in order for it to converge efficiently, unlike the existing incremental subgradient algorithm \cite{ism}.
Instead, Algorithm~\ref{alg:ngism} only needs a rough step-range as the line search automatically chooses learning rates from among this range.
Hence, it can easily be used to learn support vector machines.
Finally, Algorithm~\ref{alg:ngism} can be applied to problems in which a suitable learning rate cannot be chosen beforehand.
Hence, even if the suitable learning rate cannot be specified, line search can algorithmically find this learning rate and make proposed algorithms converge efficiently to an optimizer.

Algorithm~\ref{alg:ngism} has the following property, which is used for proving the main theorem in Section~\ref{sec:analys}.
We omit the proof of this lemma here and refer the reader to \cite[Lemma~\ref*{alg:ngism}]{arXiv}.
\begin{lem}[Fundamental Properties of Algorithm~\ref{alg:ngism} {\cite[Lemma~\ref*{alg:ngism}]{arXiv}}]\label{lem:ngism}
Let $\{x_n\}$ be a sequence generated by Algorithm~\ref{alg:ngism}.
Then, for all $y\in{}C$ and for all $n\in\nat$, the following inequality holds:
\begin{align*}
\|x_{n+1}-y\|^2
\le\|x_n-y\|^2-2\sum_{i=1}^K\lambda_{n,i}(f_i(x_n)-f_i(y))+\overline{\lambda}_n^2M^2.
\end{align*}
\end{lem}
\ifarxiv
\begin{proof}
Fix $y\in{}C$ and $n\in\nat$ arbitrarily.
From the nonexpansivity of $P_C$, the definition of subgradients, and Assumption \ref{assum:sgbdd}, we have
\begin{align*}
\|x_{n+1}-y\|^2
&=\|P_C(y_{n,K-1}-\lambda_{n,K}g_{n,K})-P_C(y)\|^2\\
&\le\|y_{n,K-1}-y-\lambda_{n,K}g_{n,K}\|^2\\
&=\|y_{n,K-1}-y\|^2-2\lambda_{n,K}\langle{}y_{n,K-1}-y,g_{n,K}\rangle+\lambda_{n,K}^2\|g_{n,K}\|^2\\
&\le\|x_n-y\|^2-2\sum_{i=1}^K\lambda_{n,i}\langle{}y_{n,i-1}-y,g_{n,i}\rangle+\sum_{i=1}^K\lambda_{n,i}^2\|g_{n,i}\|^2\\
&\le\|x_n-y\|^2-2\sum_{i=1}^K\lambda_{n,i}(f_i(y_{n,i-1})-f_i(y))+\overline{\lambda}_n^2\sum_{i=1}^KM_i^2,
\end{align*}
where the second equation comes from $\|x-y\|^2=\|x\|^2-2\langle{}x,y\rangle+\|y\|^2\quad(x,y\in\real^N)$.
Using the definition of subgradients and the Cauchy-Schwarz inequality, we have
\begin{align*}
\|x_{n+1}-y\|^2
&\le\|x_n-y\|^2-2\sum_{i=1}^K\lambda_{n,i}(f_i(x_n)-f_i(y))-2\sum_{i=1}^K\lambda_{n,i}(f_i(y_{n,i-1})-f_i(x_n))+\overline{\lambda}_n^2\sum_{i=1}^KM_i^2.\\
&\le\|x_n-y\|^2-2\sum_{i=1}^K\lambda_{n,i}(f_i(x_n)-f_i(y))+2\overline{\lambda}_n\sum_{i=1}^KM_i\|y_{n,i-1}-x_n\|+\overline{\lambda}_n^2\sum_{i=1}^KM_i^2.
\end{align*}
Further, the nonexpansivity of $P_C$ and the triangle inequality mean that, for all $i=2,3,\ldots,K$,
\begin{align*}
\|y_{n,i-1}-x_n\|
&=\|P_C(y_{n,i-2}-\lambda_{n,i-1}g_{n,i-1})-P_C(x_n)\|\\
&\le\|y_{n,i-2}-x_n-\lambda_{n,i-1}g_{n,i-1}\|\\
&\le\|y_{n,i-2}-x_n\|+\lambda_{n,i-1}\|g_{n,i-1}\|\\
&\le\|y_{n,i-2}-x_n\|+\overline{\lambda}_nM_{i-1}\\
&\le\overline{\lambda}_n\sum_{j=1}^{i-1}M_j.
\end{align*}
From above inequality and the fact that $\|y_{n,0}-x_n\|=\|x_n-x_n\|=0$, we find that
\begin{align*}
\|x_{n+1}-y\|^2
&\le\|x_n-y\|^2-2\sum_{i=1}^K\lambda_{n,i}(f_i(x_n)-f_i(y))+2\overline{\lambda}_n^2\sum_{i=1}^KM_i\sum_{j=1}^{i-1}M_j+\overline{\lambda}_n^2\sum_{i=1}^KM_i^2\\
&=\|x_n-y\|^2-2\sum_{i=1}^K\lambda_{n,i}(f_i(x_n)-f_i(y))+\overline{\lambda}_n^2\left(\sum_{i=1}^KM_i\right)^2\\
&=\|x_n-y\|^2-2\sum_{i=1}^K\lambda_{n,i}(f_i(x_n)-f_i(y))+\overline{\lambda}_n^2M^2.
\end{align*}
This completes the proof.\qed
\end{proof}
\fi

\subsection{Parallel Subgradient Algorithm}
Algorithm~\ref{alg:ngpsm} below is an extension of the parallel subgradient algorithm\cite{psm}.
\begin{algorithm}
  \caption{Parallel Subgradient Algorithm}
  \label{alg:ngpsm}
  \begin{algorithmic}[1]
    \Require{$\forall{}n\in\nat,[\underline{\lambda}_n,\overline{\lambda}_n]\subset(0,\infty)$.}
    \State{$n\gets{}1$, $x_1\in{}C$.}
    \Loop
      \ForAll{$i\in\{1,2,\ldots,K\}$}\Comment{Independently}\label{all:ngpsm:for}
        \State{$g_{n,i}\in\partial{}f_i(x_n)$.}
        \State{$\lambda_{n,i}\in[\underline{\lambda}_n,\overline{\lambda}_n]$.}\label{all:ngpsm:step}\Comment{By a line search algorithm}
        \State{$y_{n,i}:=P_C(x_n-\lambda_{n,i}g_{n,i})$.}
      \EndFor
      \State{$x_{n+1}:=\frac{1}{K}\sum_{i=1}^Ky_{n,i}$.}
      \State{$n\gets{}n+1$.}
    \EndLoop
  \end{algorithmic}
\end{algorithm}
The difference between Algorithm~\ref{alg:ngpsm} and the algorithm in \cite{psm} is Step~\ref{all:ngpsm:step} of Algorithm~\ref{alg:ngpsm}.
The existing algorithm uses a given learning rate $\lambda_n$, while Algorithm~\ref{alg:ngpsm} chooses a learning rate $\lambda_n$ from the step-range $[\underline{\lambda}_n,\overline{\lambda}_n]$ at run-time.

The common feature of Algorithm~\ref{alg:ngpsm} and the parallel subgradient algorithm \cite{psm} is loop independence (Step~\ref{all:ngpsm:for}).
This loop structure is not influenced by the computation order.
Hence, each iteration of this loop can be computed in parallel.
Therefore, parallelization using multi-core processing should be able to reduce the time needed for computing this loop procedure.
Generally speaking, the main loop of Algorithm~\ref{alg:ngpsm} is computationally heavier than the other subgradient algorithms including Pegasos, because it appends the learning rate selection (line search) procedure to the existing one.
However, parallelization alleviates this effect of the line search procedure (This is shown in Section~\ref{sec:experi}).

Next, we have the following lemma.
\begin{lem}[Fundamental Properties of Algorithm~\ref{alg:ngpsm} {\cite[Lemma~\ref*{lem:ngpsm}]{arXiv}}]\label{lem:ngpsm}
Let $\{x_n\}$ be a sequence generated by Algorithm~\ref{alg:ngpsm}.
Then, for all $y\in{}C$ and for all $n\in\nat$, the following inequality holds:
\begin{align*}
\|x_{n+1}-y\|^2
\le\|x_n-y\|^2-\frac{2}{K}\sum_{i=1}^K\lambda_{n,i}(f_i(x_n)-f_i(y))+\overline{\lambda}_n^2M^2.
\end{align*}
\end{lem}
\ifarxiv
\begin{proof}
Fix $y\in{}C$ and $n\in\nat$ arbitrarily.
From the convexity of $\|\cdot\|^2$, the nonexpansivity of $P_C$, the definition of subgradients, and Assumption \ref{assum:sgbdd}, we have
\begin{align*}
\|x_{n+1}-y\|^2
&=\left\|\frac{1}{K}\sum_{i=1}^KP_C(x_n-\lambda_{n,i}g_{n,i})-P_C(y)\right\|^2\\
&\le\frac{1}{K}\sum_{i=1}^K\left\|x_n-y-\lambda_{n,i}g_{n,i}\right\|^2\\
&=\frac{1}{K}\sum_{i=1}^K(\|x_n-y\|^2-2\lambda_{n,i}\langle{}x_n-y,g_{n,i}\rangle+\lambda_{n,i}^2\|g_{n,i}\|^2)\\
&\le\|x_n-y\|^2-\frac{2}{K}\sum_{i=1}^K\lambda_{n,i}(f_i(x_n)-f_i(y))+\overline{\lambda}_n^2M^2.
\end{align*}
This completes the proof.\qed
\end{proof}
\fi

\subsection{Line Search Algorithms}\label{ssec:lsa}
Step \ref{all:ngism:step} of Algorithm~\ref{alg:ngism} and Step~\ref{all:ngpsm:step} of Algorithm~\ref{alg:ngpsm} are implemented as line searches.
The algorithms decide an efficient learning rate $\lambda_n$ in $[\underline{\lambda}_n,\overline{\lambda}_n]$ by using $y_{n,i-1}$ in Algorithm~\ref{alg:ngism} (or $x_n$ in Algorithm~\ref{alg:ngpsm}), $g_{n,i}$, $f_i$ and other accessible information on $i$.
This is the principal idea of this paper.
We can use any algorithm that satisfies the above condition.
The following are such examples.

The simplest line search is the \emph{discrete argmin}, as shown in Algorithm~\ref{alg:dargmin}.
\begin{algorithm}
  \caption{Discrete Argmin Line Search Algorithm}
  \label{alg:dargmin}
  \begin{algorithmic}[1]
    \State{$x_p:=\begin{cases}y_{n,i-1}&\text{(Algorithm~\ref{alg:ngism})},\\x_n&\text{(Algorithm~\ref{alg:ngpsm})}\end{cases}$.}
    \State{$\lambda_{n,i}\gets{}L_1\overline{\lambda}_n+(1-L_1)\underline{\lambda}_n$.}
    \For{$L_t\in\{L_2,L_3,\ldots,L_k\}$}
      \State{$t\gets{}L_t\overline{\lambda}_n+(1-L_t)\underline{\lambda}_n$.}
      \If{$f_i(P_C(x_p-tg_{n,i}))<f_i(P_C(x_p-\lambda_{n,i}g_{n,i}))$}
        \State{$\lambda_{n,i}\gets{}t$}
      \EndIf
    \EndFor
  \end{algorithmic}
\end{algorithm}
First, we set the ratio candidates $\{L_1,L_2,\ldots,L_k\}\subset[0,1]$.
In each iteration, we compute all of the candidate objectives for the learning rate $\lambda_{n,i}=L_t\overline{\lambda}_n+(1-L_t)\underline{\lambda}_n\quad(t=1,2,\ldots,k)$ and take the best one.

Algorithm~\ref{alg:lnsel2} is a line search based on the Wolfe conditions.
\begin{algorithm}
  \caption{Logarithmic-Interval Armijo Line Search}
  \label{alg:lnsel2}
  \begin{algorithmic}[1]
    \State{$x_p:=\begin{cases}y_{n,i-1}&\text{(Algorithm~\ref{alg:ngism})},\\x_n&\text{(Algorithm~\ref{alg:ngpsm})}\end{cases}$.}
    \For{$I_R=1,1/a,1/a^2,\ldots,1/a^k$}
      \State{$\lambda_{n,i}\gets{}I_R\overline{\lambda}_n+(1-I_R)\underline{\lambda}_n$.}
      \If{$f_i(P_C(x_p-\lambda_{n,i}g_{n,i}))\le{}f_i(x_p)-c_1\langle{}x_p-P_C(x_p-\lambda_{n,i}g_{n,i}),g_{n,i}\rangle{}$}
        \State{\textbf{stop} (success).}
      \EndIf
    \EndFor
    \State{\textbf{stop} (failed).}\label{alg:lnsel2:fail}
  \end{algorithmic}
\end{algorithm}
It finds a learning rate that satisfies the sufficient decrease condition with logarithmic grids.
Once this learning rate has been found, the algorithm stops and the learning rate it found is used in the caller algorithm.
However, this algorithm may fail (Step~\ref{alg:lnsel2:fail}).
To avoid such a failure, we can make the caller algorithm use $\underline{\lambda}_n$.
This is the largest learning rate of the candidates for making an effective update of the solution.
The results of the experiments described in Section~\ref{sec:experi} demonstrate effectiveness of this algorithm\footnote{In this case, i.e., when we use Algorithm~\ref{alg:ngism} or Algorithm~\ref{alg:ngpsm} with Algorithm~\ref{alg:lnsel2}, we have to give a step-range $\{\underline{\lambda}_n,\overline{\lambda}_n\}$, a constant $c_1$ appearing in the Armijo condition, a common ratio $a$ and the number of trials $k$ of Algorithm~\ref{alg:lnsel2} as hyperparameters.}.

\section{Convergence Analysis}\label{sec:analys}
\subsection{Sequence Convergence Theorem}
Here, we first show that the limit inferiors of $\{f(x_n)\}$ generated by Algorithms~\ref{alg:ngism} and \ref{alg:ngpsm} are equal to the optimal value of the objective function $f$.
Next, we show that the generated sequence $\{x_n\}$ converges to a solution of Problem \eqref{eq:main}.
The following assumption is used to show convergence of Algorithms~\ref{alg:ngism} and \ref{alg:ngpsm}.
\begin{assum}[Step-Range Compositions]\label{assum:stpsz}
\begin{align*}
\sum_{n=1}^\infty\overline{\lambda}_n=\infty,\quad
\sum_{n=1}^\infty\overline{\lambda}_n^2<\infty,\quad
\lim_{n\to\infty}\frac{\underline{\lambda}_n}{\overline{\lambda}_n}=1,\quad
\sum_{n=1}^\infty(\overline{\lambda}_n-\underline{\lambda}_n)<\infty.
\end{align*}
\end{assum}

The following lemma states that some subsequence of the objective function value of the generated sequence converges to the optimal value.
This lemma is used to prove the main theorem described next.
\begin{lem}[Evaluation of the limit Inferior {\cite[Lemma~\ref*{lem:liminf}]{arXiv}}]\label{lem:liminf}
For a sequence $\{x_n\}$, if there exists $\alpha\in(0,\infty)$ such that, for all $y\in{}C$ and for all $n\in\nat$,
\begin{align}
\|x_{n+1}-y\|^2
\le\|x_n-y\|^2-2\alpha\sum_{i=1}^K\lambda_{n,i}(f_i(x_n)-f_i(y))+\overline{\lambda}_n^2M^2,\label{eq:basins}
\end{align}
then,
\begin{align*}
\varliminf_{n\to\infty}f(x_n)=\min_{x\in{}C}f(x).
\end{align*}
\end{lem}
\ifarxiv
\begin{proof}
Assume that $\varliminf_{n\to\infty}\sum_{i=1}^K(\lambda_{n,i}/\overline{\lambda}_n)f_i(x_n)\neq\min_{x\in{}C}f(x)$.
Then, either $\varliminf_{n\to\infty}\sum_{i=1}^K(\lambda_{n,i}/\overline{\lambda}_n)f_i(x_n)<\min_{x\in{}C}f(x)$ or $\min_{x\in{}C}f(x)<\varliminf_{n\to\infty}\sum_{i=1}^K(\lambda_{n,i}/\overline{\lambda}_n)f_i(x_n)$ holds.
First, we assume $\varliminf_{n\to\infty}\sum_{i=1}^K(\lambda_{n,i}/\overline{\lambda}_n)f_i(x_n)<\min_{x\in{}C}f(x)$.
Recall $\{x_n\}\subset{}C$ and the definition $f(x):=\sum_{i=1}^Kf_i(x)$ in the main problem~\eqref{eq:main}.
The property of the limit inferior and \cite[Exercise 4.1.31]{realan} ensure that
\begin{align*}
\min_{x\in{}C}f(x)
&\le\varliminf_{n\to\infty}f(x_n)\\
&=\varliminf_{n\to\infty}\frac{\underline{\lambda}_n}{\overline{\lambda}_n}\sum_{i=1}^Kf_i(x_n)
\end{align*}
Further, from the positivity of $f_i\quad(i=1,2,\ldots,K)$, the fact that $\underline{\lambda}_n\le\lambda_{n,i}$ and the assumption that $\varliminf_{n\to\infty}\sum_{i=1}^K(\lambda_{n,i}/\overline{\lambda}_n)f_i(x_n)<\min_{x\in{}C}f(x)$ lead to
\begin{align*}
\min_{x\in{}C}f(x)
&\le\varliminf_{n\to\infty}\sum_{i=1}^K\frac{\underline{\lambda}_n}{\overline{\lambda}_n}f_i(x_n)\\
&\le\varliminf_{n\to\infty}\sum_{i=1}^K\frac{\lambda_{n,i}}{\overline{\lambda}_n}f_i(x_n)\\
&<\min_{x\in{}C}f(x).
\end{align*}
This is a contradiction.
Next, we assume $\min_{x\in{}C}f(x)<\varliminf_{n\to\infty}\sum_{i=1}^K(\lambda_{n,i}/\overline{\lambda}_n)f_i(x_n)$ and let $\hat{y}\in\argmin_{x\in{}C}f(x)$.
Then, there exists $\varepsilon\in(0,\infty)$ such that
\begin{align*}
f(\hat{y})+2\varepsilon
=\varliminf_{n\to\infty}\sum_{i=1}^K\frac{\lambda_{n,i}}{\overline{\lambda}_n}f_i(x_n).
\end{align*}
From the definition of the limit inferior, there exists $n_0\in\nat$ such that, for all $n\in\nat$, if $n_0\le{}n$ then
\begin{align*}
\varliminf_{n\to\infty}\sum_{i=1}^K\frac{\lambda_{n,i}}{\overline{\lambda}_n}f_i(x_n)-\varepsilon
<\sum_{i=1}^K\frac{\lambda_{n,i}}{\overline{\lambda}_n}f_i(x_n).
\end{align*}
Now, $\lambda_{n,i}/\overline{\lambda}_n\le{}1$ and $0\le{}f_i(\hat{y})\quad(i=1,2,\ldots,K)$ hold.
Therefore, for all $n\in\nat$, if $n_0\le{}n$ then
\begin{align*}
\varepsilon
&=\varliminf_{n\to\infty}\sum_{i=1}^K\frac{\lambda_{n,i}}{\overline{\lambda}_n}f_i(x_n)-\varepsilon-f(\hat{y})\\
&<\sum_{i=1}^K\frac{\lambda_{n,i}}{\overline{\lambda}_n}f_i(x_n)-\sum_{i=1}^Kf_i(\hat{y})\\
&\le\sum_{i=1}^K\frac{\lambda_{n,i}}{\overline{\lambda}_n}(f_i(x_n)-f_i(\hat{y})).
\end{align*}
From inequality~\eqref{eq:basins}, for all $n\in\nat$, if $n_0\le{}n$, we have
\begin{align*}
\|x_{n+1}-\hat{y}\|^2
&\le\|x_n-\hat{y}\|^2-2\alpha\sum_{i=1}^K\lambda_{n,i}(f_i(x_n)-f_i(\hat{y}))+\overline{\lambda}_n^2M^2\\
&\le\|x_n-\hat{y}\|^2-2\alpha\overline{\lambda}_n\varepsilon+\overline{\lambda}_n^2M^2\\
&=\|x_n-\hat{y}\|^2-\overline{\lambda}_n(2\alpha\varepsilon-\overline{\lambda}_nM^2).
\end{align*}
From Assumption~\ref{assum:stpsz}, $n_1\in\nat$ exists such that $n_0\le{}n_1$ and, for all $n\in\nat$, if $n_1\le{}n$, $\overline{\lambda}_n\le\alpha\varepsilon/M^2$.
Hence, if $n_1\le{}n$, we have
\begin{align*}
0
&\le\|x_{n+1}-\hat{y}\|^2\\
&\le\|x_n-\hat{y}\|^2-\alpha\varepsilon\overline{\lambda}_n\\
&\le\|x_{n_1}-\hat{y}\|^2-\alpha\varepsilon\sum_{k=n_1}^n\overline{\lambda}_k.
\end{align*}
for all $n\in\nat$.
From Assumption~\ref{assum:stpsz}, the right side diverges negatively, which is a contradiction.
Overall, we have
\begin{align*}
\varliminf_{n\to\infty}\sum_{i=1}^K\frac{\lambda_{n,i}}{\overline{\lambda}_n}f_i(x_n)
=\min_{x\in{}C}f(x).
\end{align*}
Next, let us assume that $\varliminf_{n\to\infty}\sum_{i=1}^K(\lambda_{n,i}/\overline{\lambda}_n)f_i(x_n)\neq\varliminf_{n\to\infty}f(x_n)$.
Now, $\lambda_{n,i}/\overline{\lambda}_n\le{}1$ and $0\le{}f_i(x_n)\quad(i=1,2,\ldots,N)$ hold for all $n\in\nat$.
Therefore, we have
\begin{align*}
\varliminf_{n\to\infty}\sum_{i=1}^K\frac{\lambda_{n,i}}{\overline{\lambda}_n}f_i(x_n)
\le\varliminf_{n\to\infty}f(x_n).
\end{align*}
Hence, from the positivity of $f_i\quad(i=1,2,\ldots,K)$, the fact that $\underline{\lambda}_n\le\lambda_{n,i}$, and \cite[Exercise 4.1.31]{realan}, we have
\begin{align*}
\varliminf_{n\to\infty}f(x_n)
&=\varliminf_{n\to\infty}\frac{\underline{\lambda}_n}{\overline{\lambda}_n}f(x_n)\\
&\le\varliminf_{n\to\infty}\sum_{i=1}^K\frac{\lambda_{n,i}}{\overline{\lambda}_n}f_i(x_n)\\
&<\varliminf_{n\to\infty}f(x_n).
\end{align*}
However, this is a contradiction.
This completes the proof.\qed
\end{proof}
\fi

The following is the main theorem of this paper.
\begin{thm}[Main Theorem]
The sequence $\{x_n\}$ generated by Algorithm~\ref{alg:ngism} or \ref{alg:ngpsm} converges to an optimal solution to the main problem \eqref{eq:main}.
\end{thm}
\begin{proof}
Let $\hat{y}\in\argmin_{x\in{}C}f(x)$ and fix $n\in\nat$.
From Lemmas~\ref{lem:ngism} and \ref{lem:ngpsm}, there exists $\alpha\in(0,\infty)$ such that
\begin{align*}
\|x_{n+1}-\hat{y}\|^2
\le\|x_n-\hat{y}\|^2-2\alpha\sum_{i=1}^K\lambda_{n,i}(f_i(x_n)-f_i(\hat{y}))+\overline{\lambda}_n^2M^2.
\end{align*}
From $0\le{}f_i(\hat{y}),f_i(x_n)\quad(i=1,2,\ldots,K)$, we have
\begin{align}
\|x_{n+1}-\hat{y}\|^2
&\le\|x_n-\hat{y}\|^2-2\alpha\underline{\lambda}_n\sum_{i=1}^Kf_i(x_n)+2\alpha\overline{\lambda}_n\sum_{i=1}^Kf_i(\hat{y})+\overline{\lambda}_n^2M^2\nonumber\\
&=\|x_n-\hat{y}\|^2-2\alpha\underline{\lambda}_n\sum_{i=1}^K(f_i(x_n)-f_i(\hat{y}))+2\alpha(\overline{\lambda}_n-\underline{\lambda}_n)\sum_{i=1}^Kf_i(\hat{y})+\overline{\lambda}_n^2M^2\nonumber\\
&\le\|x_n-\hat{y}\|^2+2\alpha{}f(\hat{y})(\overline{\lambda}_n-\underline{\lambda}_n)+\overline{\lambda}_n^2M^2\label{eq:daiji}\\
&\le\|x_1-\hat{y}\|^2+2\alpha{}f(\hat{y})\sum_{i=1}^n(\overline{\lambda}_i-\underline{\lambda}_i)+M^2\sum_{i=1}^n\overline{\lambda}_i^2.\nonumber
\end{align}
From Assumption~\ref{assum:stpsz}, the left side of the above inequality is bounded.
Hence, $\{x_n\}$ is bounded.
From Lemma~\ref{lem:liminf}, a subsequence $\{x_{n_i}\}\subset\{x_n\}$ and $u\in\argmin_{x\in{}C}f(x)$ exist such that $x_{n_i}\to{}u$.
Using \cite[Lemma~1.7.(ii)]{itvm} with inequality~\eqref{eq:daiji}, this implies $x_n\to{}u$.
This completes the proof.\qed
\end{proof}

\subsection{Convergence Rates}
To show the convergence rates of Algorithms~\ref{alg:ngism} and \ref{alg:ngpsm}, we assume $\underline{\lambda}_n:=\overline{\lambda}_n:=1/n$ for all $n\in\nat$.
We also assume the existence of $\mu\in(0,\infty)$ such that
\begin{align}
f(x)-f(\hat{y})\ge\mu\|x-\hat{y}\|^2\quad\left(x\in{}C,\hat{y}\in\argmin_{u\in{}C}f(u)\right).\label{eqn:wsm}
\end{align}
The strong convexity of $f$ implies Condition~\eqref{eqn:wsm}\cite[Inequality~(16)]{cvrate}.
First, We give the following lemma, which is required to show the convergence rates of Algorithms~\ref{alg:ngism} and \ref{alg:ngpsm}.
\begin{lem}{\rm (\cite[Lemma~2.1]{cvrate}, \cite[Lemma~4]{polyak})}\label{lem:tanaka}
Let $\{u_n\}\subset[0,\infty)$ be such that
\begin{align*}
u_{n+1}\le\left(1-\frac{p}{n}\right)u_n+\frac{d}{n^2}\quad(n\in\nat)
\end{align*}
for some $p,d\in(0,\infty)$.
Then
\begin{align*}
\begin{cases}
u_n=O(\frac{1}{n^p}) & (p < 1), \\
u_n=O(\frac{\log{}n}{n}) & (p = 1), \\
u_n\le{}\frac{d}{n(p-1)}+o(\frac{1}{n}) & (p > 1).
\end{cases}
\end{align*}
\end{lem}
Next, we prove two propositions that show the convergence rates of Algorithms~\ref{alg:ngism} and \ref{alg:ngpsm}.
\begin{prop}[Convergence Rate of Algorithm~\ref{alg:ngism}]
Let $\{x_n\}$ be a sequence generated by Algorithm~\ref{alg:ngism} and $\hat{y}\in\argmin_{y\in{}C}f(y)$.
Then, the following hold:
\begin{align*}
\begin{cases}
\|x_{n+1}-\hat{y}\|=O(\frac{1}{n^{2\mu}}) & (2\mu < 1), \\
\|x_{n+1}-\hat{y}\|=O(\frac{\log{}n}{n}) & (2\mu = 1), \\
\|x_{n+1}-\hat{y}\|\le{}\frac{M^2}{n(2\mu-1)}+o(\frac{1}{n}) & (2\mu > 1).
\end{cases}
\end{align*}
\end{prop}
\begin{proof}
From Lemma~\ref{lem:ngism} and inequality~\eqref{eqn:wsm}, we have
\begin{align*}
\|x_{n+1}-\hat{y}\|^2
\le\left(1-\frac{2\mu}{n}\right)\|x_n-\hat{y}\|^2+\frac{M^2}{n^2}.
\end{align*}
for all $n\in\nat$.
Lemma~\ref{lem:tanaka} with $p:=2\mu, d:=M^2$ completes the proof.
\qed
\end{proof}
This result implies that Algorithm~\ref{alg:ngism} is in the same class of convergence efficiency as the incremental subgradient algorithm\cite[Proposition~2.8]{cvrate}.

\begin{prop}[Convergence Rate of Algorithm~\ref{alg:ngpsm}]
Let $\{x_n\}$ be a sequence generated by Algorithm~\ref{alg:ngpsm} and $\hat{y}\in\argmin_{y\in{}C}f(y)$.
Then, the following hold:
\begin{align*}
\begin{cases}
\|x_{n+1}-\hat{y}\|=O(\frac{1}{n^{2\mu/K}}) & \left(\mu < \frac{K}{2}\right), \\
\|x_{n+1}-\hat{y}\|=O(\frac{\log{}n}{n}) & \left(\mu = \frac{K}{2}\right), \\
\|x_{n+1}-\hat{y}\|\le{}\frac{M^2}{n(2\mu/K-1)}+o(\frac{1}{n}) & \left(\mu > \frac{K}{2}\right).
\end{cases}
\end{align*}
\end{prop}
\begin{proof}
From Lemma~\ref{lem:ngpsm} and inequality~\eqref{eqn:wsm}, we have
\begin{align*}
\|x_{n+1}-y\|^2
\le\left(1-\frac{2\mu}{nK}\right)\|x_n-\hat{y}\|^2+\frac{M^2}{n^2}.
\end{align*}
for all $n\in\nat$.
Lemma~\ref{lem:tanaka} with $p:=2\mu/K, d:=M^2$ completes the proof.
\qed
\end{proof}

To above analyses assumed $\underline{\lambda}_n=\overline{\lambda}_n$.
However, Algorithms~\ref{alg:ngism} and \ref{alg:ngpsm} can use different values of $\underline{\lambda}_n$ and $\overline{\lambda}_n$.
This implies that Algorithms~\ref{alg:ngism} and \ref{alg:ngpsm} may converge faster than theoretical rates given here.

\section{Experiments}\label{sec:experi}
In this section, we present the results of experiments evaluating our algorithms and comparing them with the existing algorithms.
For our experiments, we used a MacPro (Late 2013) computer with a 3GHz 8-Core Intel Xeon E5 CPU, 32GB 1866MHz DDR3 memory, and 500GB flash storage.
The operating system was MacOS Sierra (version 10.12.6). The experimental codes were written in Python 3.6 and ran on the CPython implementation.

\subsection{Validation of Convergence Properties with a Simple Problem}
A concrete test problem with a closed-form solution is a good way to evaluate the performance of algorithms in detail \cite{tutsoy2016-2,tutsoy2016}.
Here, we used the existing and proposed algorithms to solve a simple problem.
The goals were to compare their performances under equal conditions, to use the best parameters for each algorithm calculated from the theoretical analyses, and to evaluate these algorithms with the detailed indicators such as the distance between an acquired solution and the actual solution of the test problem.
The test problem is as follows.
\begin{prob}[Test problem]\label{prob:test}
  Let $f_i(x):=(i+1)x_i^2\ (x\in\real^N; i=1,2,\ldots,n)$, $c\in\real^N$, and $r\in\real$.
  Then, we would like to
  \begin{align*}
    \text{minimize } & f(x):=\sum_{i=1}^Kf_i(x) \\
    \text{subject to } & \norm{x - c} \le r \text{ and }x_i=0\ (i=3,4,\ldots,N).
  \end{align*}
\end{prob}
This problem is obviously an instance of Problem~\ref{eq:main}.
Of course, the continuity of the objective function and the boundedness of the constraint ensure the above problem satisfies Assumption~\ref{assum:sgbdd}.
We set $c:=(2, 1, 0, \ldots, 0)^\top$ and $r:=1$.
The optimal solution is accordingly $x^\star=((2-/\sqrt{2})/2, (2-\sqrt{2})/2, 0, \ldots, 0)$.

We set the number of dimensions to $N:=16$, i.e., equal to the number of logical cores of the experimental computer.
We gave $x_1:=(2, 1, 0, \ldots, 0)^\top$, the center of the feasible set, as an initial point.
We selected the incremental and parallel subgradient algorithms for comparison.
These algorithms use a priori given learning rates; that is, they coincide with Algorithms~\ref{alg:ngism}, \ref{alg:ngpsm} with the settings $\underline{\lambda}_n=\overline{\lambda}_n=\lambda_n\in(0,\infty)\ (n\in\nat)$.
In this comparison, we gave learning rates of $\lambda_n:=1/(nN^2)$, which are appropriately chosen based on the following proposition related to the existing algorithms.
\begin{prop}[{\cite[Lemma~2.1]{iiduka2015}, \cite[Lemma~3.1]{yamada2005}}]
  Suppose that $f:\real^N\to\real$ is $c$-strongly convex and differentiable, $\nabla f:\real^N\to\real^N$ is $L$-Lipschitz continuous, and $\mu\in(0,2c/L^2)$.
  Define $T:\real^N\to\real^N$ by $T(x):=x-\mu\nabla f(x)\ (x\in\real^N)$.
  Then, $T$ is a contractive mapping.
\end{prop}
We set $\underline{\lambda}_n:=100/((n+10000)N^2),\overline{\lambda}_n:=100/(nN^2)$ for each $n\in\nat$ as the parameters of the proposed algorithms.
This step-range contains the learning rates of the existing algorithms.
We used Algorithm~\ref{alg:lnsel2} with the parameters $c_1:=0.99$, $a:=0.5$, and $k:=7$. We limited the iterations to 1,000 and evaluated the following indicators:
\begin{itemize}
  \item $F_n$: value of the objective function, i.e., $F_n:=\sum_{i=1}^nf_i(x_n)$,
  \item $D_n$: distance to the optimal solution, i.e., $D_n:=\norm{x^\star-x_n}$,
  \item $T_n$: running time of the algorithm.
\end{itemize}
The behaviors of $\{F_n\}$ and $\{D_n\}$ in each iteration $n$ are shown in Figure~\ref{fig:ez}.
\begin{figure}[htbp]
  \centering
  \subcaptionbox{Behavior of $F_n$ in each iteration $n$}{\includegraphics[width=0.45\textwidth]{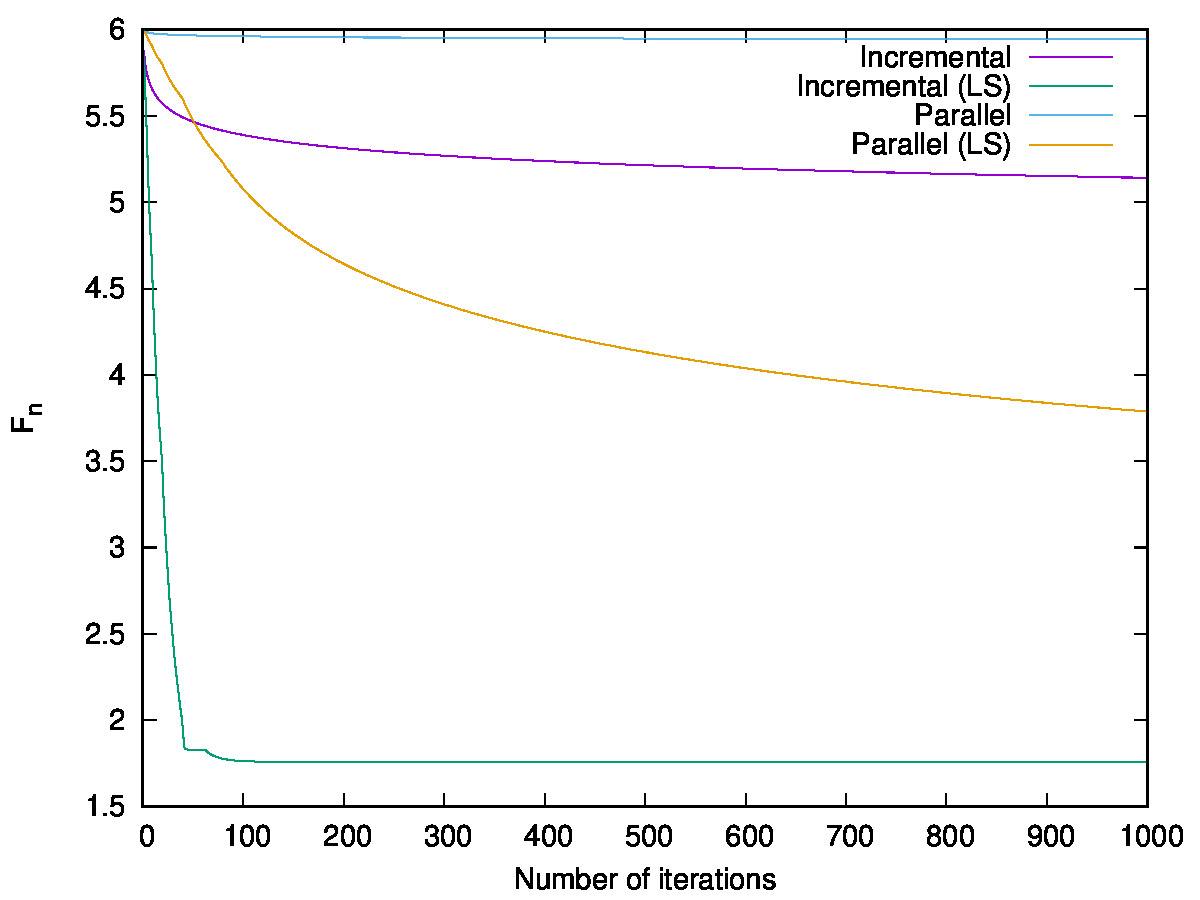}}
  \subcaptionbox{Behavior of $D_n$ in each iteration $n$}{\includegraphics[width=0.45\textwidth]{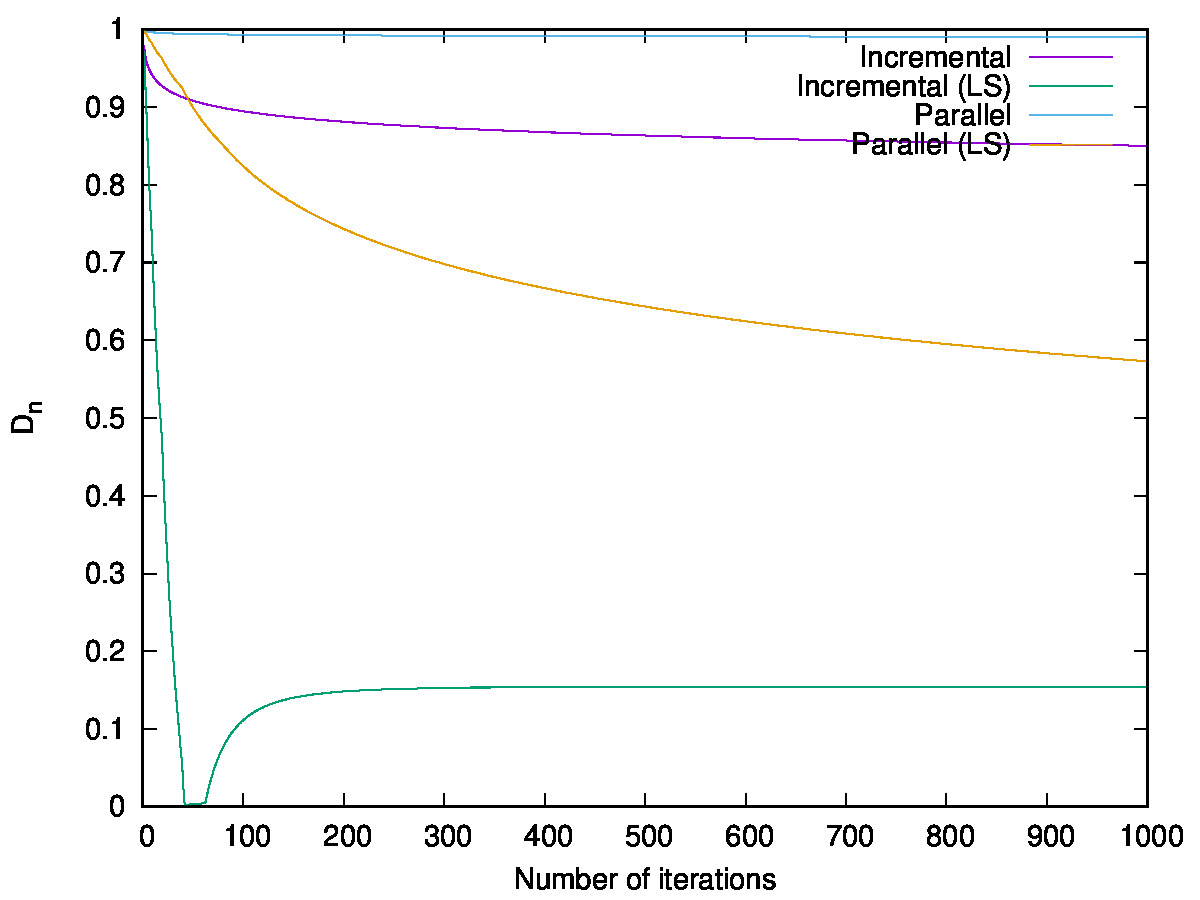}}
  \caption{Numerical comparison of running the existing and proposed algorithms on the test problem~\ref{prob:test}}\label{fig:ez}
\end{figure}
The result of Algorithm~\ref{alg:ngism} dropped to the optimal dramatically in terms of both $\{F_n\}$ and $\{D_n\}$ within the first fifty iterations.
The graphs of the other algorithms decreased similarly, but those of the algorithms with the line search decreased faster.
Table~\ref{tbl:rt} lists the running times of the existing and proposed algorithms for 1,000 iterations.
\begin{table}[htbp]
  \caption{Running time of each algorithm in solving the test problem~\ref{fig:ez}}\label{tbl:rt}
  \begin{tabular}{ll}
    \toprule
    Algorithm & Running time ($T_{1000}$ [s]) \\
    \midrule
    Incremental subgradient algorithm & 0.34341614 \\
    Algorithm~\ref{alg:ngism} & 0.97763861 \\
    Parallel subgradient algorithm & 0.11232432 \\
    Algorithm~\ref{alg:ngpsm} & 0.24512658 \\
    \bottomrule
  \end{tabular}
\end{table}
Compared with the existing algorithms, the proposed algorithms needed a bit more time for running.
However, they dramatically reduced the value of the objective function.
Therefore, they converged faster that the existing algorithms.

\subsection{Comparison of the existing and proposed algorithms in the task of learning with support vector machines}\label{subsec:1}
This subsection compares Algorithms~\ref{alg:ngism} and \ref{alg:ngpsm} with Pegasos \cite{pegasos}.
To evaluate their performance, we applied them to the following learning task.
\begin{prob}
[The task of learning with a support vector machine {\cite{pegasos}}]\label{prob:svm} Let $C$ be a positive real number.
Given a training set $\{(x_i,y_i)\}$, where $x_i\in\real^N\quad(i=1,2,\ldots,K)$ and $y_i\in\{1, -1\}\quad(i=1,2,\ldots,K)$, we would like to
\begin{align*}
\text{minimize } & f(w):=\frac{1}{C}\norm{w}^2+\frac{1}{K}\sum_{i=1}^K\max\{0, 1-y_i\ip{w,x_i}\}\\
\text{subject to } & w\in{}X:=\{w:\norm{w}\le\sqrt{C}\}.
\end{align*}
\end{prob}
This optimization problem is introduced in \cite{pegasos} for learning with a support vector machine.
The first term of the objective function is a penalty term that depends on the constraint set, and the second term is a loss function.
The loss function returns higher values if the learner $w$ can not classify an instance $(x_i,y_i)$ correctly.
The norm value of the learner $w$ does not affect the classification results due to the immutability of the signs of the decision function $\ip{w,x_i}$.
Therefore, we can limit this value to a constant $C$.
Now, let $f_i(w):=((1/C)\norm{w}^2+\max\{0,1-y_i\ip{w,x_i}\})/K$.
Then, $f=\sum_{i=1}^Kf_i$ holds and Problem~\ref{prob:svm} can be handled as an instance of Problem~\eqref{eq:main}.

We used the machine learning datasets shown in Table~\ref{tbl:datasets}.
\begin{table}[htbp]
\caption{Datasets used in our experiments}
\label{tbl:datasets}
\begin{tabular}{lllll}
\toprule
Name & \#Instances & \#Attributes & Missing Values & Attribute Characteristics \\
\midrule
Iris (binary class) & 100 & 4 & No & Real \\
Iris (multiclass) & 150 & 4 & No & Real \\
Australian & 590 & 14 & No & Real \\
Horse-colic & 368 & 27 & Yes & Categorical, Integer, Real \\
Breast-cancer-wisconsin & 699 & 10 & Yes & Integer \\
Census-income & 48842 & 14 & Yes & Categorical, Integer \\
Internet-advertisements & 3279 & 1558 & Yes & Categorical, Integer, Real \\
MNIST & 14780 & 784 & No & Integer \\
RANDOM1 & 20 & 100 & No & Real \\
RANDOM2 & 200 & 1000 & No & Real \\
\bottomrule
\end{tabular}
\end{table}
The ``australian'' data set is from LIBSVM Data \cite{libsvm}.
The ``MNIST'' data set contains handwritten ``0'' and ``1'' digits and is provided by \cite{mnist}. The ``RANDOM1'' and ``RANDOM2'' datasets were generated using the \texttt{sklearn.datasets.make\_classification} function with a fixed \texttt{random\_state}.
The others are from the UCI Machine Learning Repository \cite{uci}.
The number of classes of ``iris (multiclass)'' is three, and the others are binary classification datasets.

Missing values were complemented by using the \texttt{sklearn.impute.SimpleImputer} class.
Categorical attributes were binarized using the \texttt{sklearn.preprocessing.OneHotEncoder} class.
Each data set was scaled using the \texttt{sklearn.preprocessing.StandardScaler} class.
These preprocessing methods and classes are from the scikit-learn \cite{scikit-learn} package for Python3.

The Pegasos algorithm used for this comparison is listed as Algorithm~\ref{alg:pegasos}.
\begin{algorithm}
\caption{Pegasos \cite[Fig.~1]{pegasos}}
\label{alg:pegasos}
\begin{algorithmic}[1]
    \State{$n\gets{}1, w_1\in{}X$.}
    \Loop{}
\State{$i_n\in\{1,2,\ldots,K\}$.}\Comment{Chosen uniformly at random}
      \State{$g\in\partial{}f_{i_n}(w_n)$.}
      \State{$\lambda_n:=C/n$.}
      \State{$w_{n+1}\gets{}P_X(w_n-\lambda_ng_n)$.}
      \State{$n\gets{}n+1$.}
    \EndLoop{}
  \end{algorithmic}
\end{algorithm}

We set $C:=10^{-1}$ and gave $\lambda_n:=10^{-1}/(n+10^{8})$ to Algorithms~\ref{alg:ngism} and \ref{alg:ngpsm}.
We used Algorithm~\ref{alg:lnsel2} with $c_1:=0.99$ for the line search step in Algorithms~\ref{alg:ngism} and \ref{alg:ngpsm}.
The main loops in Algorithms~\ref{alg:ngism} and \ref{alg:pegasos} were iterated $100K$ times, while the main loop in Algorithm~\ref{alg:ngpsm} was iterated $100$ times.
This setting means that the algorithms could refer to each of the functions $f_i\,(i=1,2,\ldots,K)$ 1000 times.

We added scores of the SMO algorithm, one of the major algorithms for learning with a support vector machine, to the experimental results for each dataset.
We used the implementation of the SMO algorithm in Python \footnote{\url{https://github.com/LasseRegin/SVM-w-SMO}} for calculating these scores.

First, let us look at the results for the iris (binary class) data set.
Table~\ref{tbl:iris2} lists the computational times for learning, the classification scores on the training and test sets, and the values of the objective function.
\begin{table}
\caption{Iris (binary class)}
\label{tbl:iris2}
\begin{tabular}{lllll}
\toprule
Algorithm & Time [sec] & Score (Training) & Score (Test) & Objective \\
\midrule
Pegasos & 0.12741225 & 1.00000000 & 1.00000000 & 0.95967555 \\
Algorithm~\ref{alg:ngism} & 0.28701049 & 1.00000000 & 1.00000000 & 0.94237229 \\
Algorithm~\ref{alg:ngpsm} & 0.03734168 & 1.00000000 & 1.00000000 & 0.91133305 \\
\midrule
SMO Algorithm & 0.00515115 & 1.00000000 & 1.00000000 & -- \\
\bottomrule
\end{tabular}
\end{table}
We used the \texttt{sklearn.model\_selection.train\_test\_split} method provided by the scikit-learn package \cite{scikit-learn} to split the dataset into training and test sets.
The number of instances in the training set was 30 and the number of instances in the test set was 70.
The results indicate that Algorithm~\ref{alg:ngpsm} performed better than Pegasos and Algorithm~\ref{alg:ngism} in terms of computational time and value of the objective function.
In addition, Algorithm~\ref{alg:ngism} worked out a better approximation than Pegasos did in terms of the objective function.
Hence, Algorithms~\ref{alg:ngism} and \ref{alg:ngpsm} ran more efficiently than the existing algorithm.
However, the SMO algorithm ran more quickly than the other algorithms, while keeping the highest score.

Next, let us look at the results of the multiclass classification using the iris (multiclass) dataset.
Table~\ref{tbl:iris3} lists the computational times for learning and the classification scores on the training and test sets.
\begin{table}
\caption{Iris (multiclass; Algorithms~\ref{alg:ngism}, \ref{alg:ngpsm}, and \ref{alg:pegasos} are used as solvers for the subproblem appearing in this multiclass classification experiment.)}
\label{tbl:iris3}
\begin{tabular}{lllll}
\toprule
Algorithm & Time [sec] & Score (Training) & Score (Test) & Avg. Objective \\
\midrule
Pegasos & 0.59731907 & 0.77777778 & 0.79047619 & 0.98524879 \\
Algorithm~\ref{alg:ngism} & 1.30901892 & 0.77777778 & 0.80952381 & 0.97505393 \\
Algorithm~\ref{alg:ngpsm} & 0.08996129 & 0.80000000 & 0.81904762 & 0.95314453 \\
\midrule
SMO Algorithm & 0.05340354 & 0.80000000 & 0.82857143 & -- \\
\bottomrule
\end{tabular}
\end{table}
We used the \texttt{sklearn.model\_selection.train\_test\_split} method provided by the scikit-learn package \cite{scikit-learn} to split the dataset into training and test sets.
To construct multiclass classifiers from Algorithm~\ref{alg:ngism}, \ref{alg:ngpsm}, and \ref{alg:pegasos}, we used \texttt{sklearn.multiclass.OneVsRestClassifier} class which provides a construction of one-versus-the-rest (OvR) multiclass classifiers.
In this experiment, the number of instances in the training set was 45 and the number of instances in the test set was 105.
The results show that Algorithms~\ref{alg:ngism} and \ref{alg:ngpsm} performed better than Pegasos with respect to their scores for the training and test sets.
In addition, the computational time of Algorithm~\ref{alg:ngpsm} was shorter than those of Pegasos and Algorithm~\ref{alg:ngism}.
In this case, Algorithm~\ref{alg:ngpsm} learned a classifier whose classification score is similar to the one of the SMO algorithm in almost same running time.

To compare the algorithms in detail, we conducted experiments on other datasets: australian, horse-colic, breast-cancer-wisconsin, census-income, internet-advertisements, MNIST, RANDOM1 and RANDOM2.
We performed a stratified five-fold cross-validation with the \texttt{sklearn.}\texttt{model\_selection.}\\{}\texttt{StratifiedKFold} class.
Table~\ref{tbl:results} shows the averages of the computational times for learning, the classification scores on the test sets, and the values of the objective function for each dataset.
TLE (time limit exceeded) in the table means that the experiment was compulsorily terminated because the running time of the SMO algorithm excessively exceeded those of the other algorithms.
The classification scores are calculated using the following formula implemented as the \texttt{sklearn.base.ClassifierMixin.score} method,
\begin{align*}
\text{(Score)}:=\frac{\text{(\#Accurate Instances)}}{\text{(\#Instances)}}.
\end{align*}
This value is an increasing evaluation of goodness of fit \cite[Section~4]{scikit-learn}.

\begin{landscape}
  \begin{table}[H]
    \caption{Averages of computational times for learning, classification scores on the test sets, and values of the objective function for each dataset}\label{tbl:results}
    \begin{tabular}{l|lll|lll}
      \toprule
                                & \multicolumn{3}{l|}{Australian}      & \multicolumn{3}{l}{Horse-colic} \\
      Algorithm                 & Time [sec] & Score      & Objective  & Time [sec] & Score      & Objective \\
      \midrule
      Pegasos                   & 2.42488674 & 0.82920957 & 0.99754775 & 1.54011672 & 0.71261261 & 0.99908892 \\
      Algorithm~\ref{alg:ngism} & 6.04898934 & 0.84939531 & 0.99070582 & 3.97349780 & 0.70713213 & 0.99907467 \\
      Algorithm~\ref{alg:ngpsm} & 0.06925168 & 0.85227300 & 0.95032527 & 0.05841917 & 0.72620120 & 0.96485216 \\
      \midrule
      SMO Algorithm             & 1.63863547 & 0.86238696 & --         & 4.96065068 & 0.71193694 & -- \\
      \bottomrule
      \toprule
                                & \multicolumn{3}{l|}{Breast-cancer-wisconsin} & \multicolumn{3}{l}{Census-income} \\
      Algorithm                 & Time [sec] & Score      & Objective  & Time [sec] & Score      & Objective \\
      \midrule
      Pegasos                   & 2.49116932 & 0.96843767 & 0.99228738 & 180.51582360 & 0.70191638 & 0.99995947 \\
      Algorithm~\ref{alg:ngism} & 6.12068390 & 0.96558053 & 0.95931909 & 456.40569000 & 0.71029029 & 0.99909114 \\
      Algorithm~\ref{alg:ngpsm} & 0.06906789 & 0.96558053 & 0.82970374 &   0.52751211 & 0.71031078 & 0.96254982 \\
      \midrule
      SMO Algorithm             & 0.98688517 & 0.96989708 & --         & TLE          & --         & --         \\
      \bottomrule
      \toprule
                                & \multicolumn{3}{l|}{Internet-advertisements} & \multicolumn{3}{l}{MNIST} \\
      Algorithm                 & Time [sec]  & Score      & Objective  & Time [sec]   & Score      & Objective \\
      \midrule
      Pegasos                   & 17.90985671 & 0.95730497 & 0.99954933 &  69.15901990 & 0.98687356 & 0.99894996 \\
      Algorithm~\ref{alg:ngism} & 44.18707687 & 0.59436744 & 0.99972300 & 153.40252700 & 0.99032472 & 0.99827819 \\
      Algorithm~\ref{alg:ngpsm} &  0.30783534 & 0.95517036 & 0.87879677 &   0.87765352 & 0.99269253 & 0.60622818 \\
      \midrule
      SMO Algorithm             & TLE         & --         & --         & TLE         & --         & --          \\
      \bottomrule
      \toprule
                                & \multicolumn{3}{l|}{RANDOM1}          & \multicolumn{3}{l}{RANDOM2} \\
      Algorithm                 & Time [sec]  & Score      & Objective  & Time [sec]   & Score      & Objective \\
      \midrule
      Pegasos                   & 0.34232559 & 0.88000000 & 0.99113491  &  3.99642010 & 0.84699855 & 0.99941910 \\
      Algorithm~\ref{alg:ngism} & 1.38527165 & 0.86000000 & 0.99789224  & 17.21619467 & 0.68499385 & 0.99982958 \\
      Algorithm~\ref{alg:ngpsm} & 0.03952229 & 0.90000000 & 0.93026020  &  0.04702928 & 0.87099875 & 0.95740116 \\
      \midrule
      SMO Algorithm             & 0.09197520 & 0.84000000 & --          & 16.82656507 & 0.79700312 & --         \\
      \bottomrule
    \end{tabular}
  \end{table}
\end{landscape}

Let us evaluate the computational times for learning, the classification scores on the test sets, and the values of the objective function in order.
For a detailed, fair, statistical comparison, we used an analysis of variance (ANOVA) test and Tukey--Kramer's honestly significant difference (HSD) test.
We used the \texttt{scipy.stats.f\_oneway} method in the SciPy library as the implementation of the ANOVA tests and the \texttt{statsmodels.stats.multicomp.pairwise\_tukeyhsd} method in the StatsModels package as the implementation of Tukey--Kramer's HSD test.
The ANOVA test examines whether the hypothesis that the given groups have the same population mean is rejected or not.
Therefore, we can use it for finding an experimental result that has a significant difference. Tukey--Kramer's HSD test can be used to find specifically which pair has a significant difference in groups.
We set 0.05 (5\%) as the significance level for the ANOVA and Tukey--Kramer's HSD tests and used the results of each fold of the cross-validation for the statistical evaluations described below.

First, we consider the computation times for learning.
All $p$-values computed by the ANOVA tests were much less than 0.05; this range was from $10^{-26}$ to $10^{-8}$.
This implies that a significant difference exists in terms of the computation time between the algorithms for every dataset.
The results of the Tukey--Kramer's HSD tests showed that the computation times of Algorithm~\ref{alg:ngpsm} for the australian, horse-colic, breast-cancer-wisconsin, census-income, internet-advertisements, and MNIST datasets were significantly shorter than those of Pegasos, Algorithm~\ref{alg:ngism} and the SMO algorithm.
However, the null hypotheses about Algorithm~\ref{alg:ngpsm} and the SMO algorithm for the RANDOM1 dataset, and Algorithm~\ref{alg:ngpsm} and Pegasos for the RANDOM2 dataset were not rejected.
Therefore, for most of the practical datasets, Algorithm~\ref{alg:ngpsm} runs significantly faster than the existing algorithms.
However, it seems that there are a few cases where the computation time of the Algorithm~\ref{alg:ngpsm} roughly equals those of the existing algorithms.

Next, we consider the classification scores on the test sets.
The ANOVA tests indicate that significant differences may exist in the census-income, internet-advertisements, and RANDOM2 datasets.
However, Tukey-Kramer's HSD test could not reject the null hypotheses between any two algorithms for the census-income dataset.
The results of the Tukey-Kramer's HSD tests showed that the scores of Algorithm~\ref{alg:ngism} were significantly worse than those of the other algorithms for the internet-advertisements and RANDOM2 datasets.
Moreover, they showed that the scores of Algorithm~\ref{alg:ngpsm} were significantly better than those of Algorithm~\ref{alg:ngism} and the SMO algorithm for the RANDOM2 dataset.
Although each algorithm may have advantages or disadvantages compared with the others on certain datasets, the classification scores of the four algorithms were roughly similar as a whole.

Next, we consider the values of the objective function.
All $p$-values computed by the ANOVA tests were much less than 0.05; this range was from $10^{-32}$ to $10^{-12}$.
This implies that a significant difference exists in terms of the values of the objective function between the algorithms for every dataset.
The results of the Tukey-Kramer's HSD tests showed that the values of the objective function of Algorithm~\ref{alg:ngpsm} were significantly lower than those of Pegasos and Algorithm~\ref{alg:ngism} for all datasets.
Therefore, Algorithm~\ref{alg:ngism} reduced the value of objective function more than the other algorithms.

Figure~\ref{fig:boxplot} illustrates a box-plot comparison of Pegasos, Algorithm~\ref{alg:ngism}, Algorithm~\ref{alg:ngpsm}, and the SMO Algorithm in terms of classification scores on the test sets.
We used the results of all folds of the cross-validations and all datasets shown in Table 5 for making this box-plot comparison.
\begin{figure}[htbp]
  \centering
  \includegraphics[width=4truein]{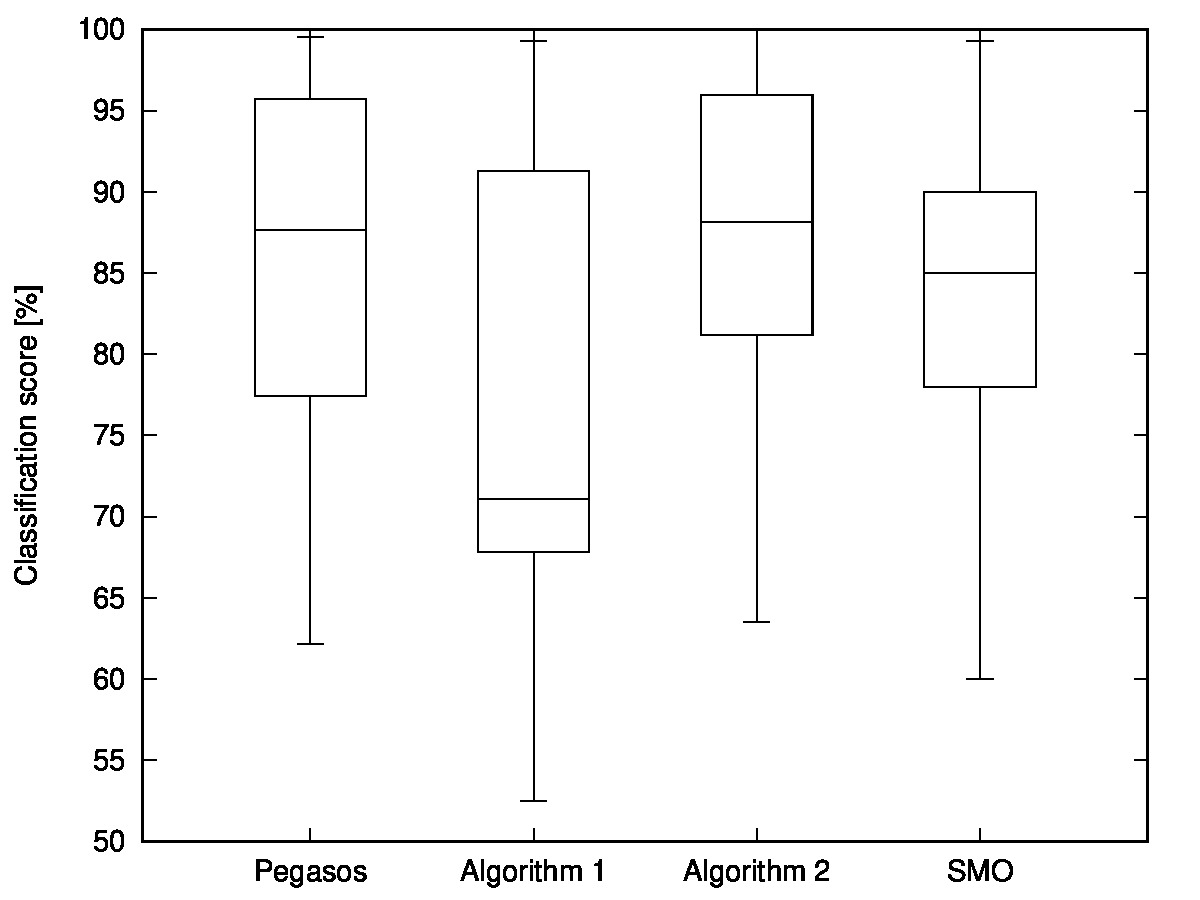}
  \caption{Box-plot comparison of Pegasos, Algorithm~\ref{alg:ngism}, Algorithm~\ref{alg:ngpsm} and SMO Algorithm in terms of classification scores on the test sets}\label{fig:boxplot}
\end{figure}
The horizontal lines in the boxes represent the median scores, and the boxes represent the upper and lower quartiles of the resulting scores.
Similar to the above discussion of the average, we find that Algorithm 2 has the best median of the classification scores among the four algorithms.
The results of Pegasos were similar to those of Algorithm~\ref{alg:ngpsm}; however, the computation time of Algorithm~\ref{alg:ngpsm} was dramatically shorter than that of Pegasos.
Therefore, box-plot comparison also shows that Algorithm~\ref{alg:ngpsm} is the most useful method for learning with a support vector machine.

In conclusion, the above comparison indicates that, whichever algorithm we use, we can obtain classifiers whose classification abilities are similar.
However, Algorithm~\ref{alg:ngpsm} runs faster than the other algorithms, and it reduces the value of the objective function more.
Therefore, the series of experiments and considerations lead us to conclude that Algorithm~\ref{alg:ngpsm} is useful for learning with a support vector machine.

\subsection{Application to learning multilayer neural networks}
Let us consider using the proposed algorithms to learn a multilayer neural network with.
Our algorithms are not limited to being used for learning support vector machines; they can also be used for optimizing general functions.
Therefore, we can also use them for learning a multilayer neural network.
Here, we should note that the incremental subgradient algorithm is a specialization of the stochastic subgradient algorithm, which is a useful algorithm for learning a neural network.
Hence, we decided to apply it to a concrete task for learning a multilayer neural network and evaluate its applicability to learning deep neural networks.

We used the MNIST database \cite{mnist} of handwritten digits for this experiment.
The goal is recognizing what Arabic numerals are written on the given images.
To achieve this goal, we can use 60,000 examples contained in the training set.
Each example is composed of a $28\times 28$ image that expresses a handwritten digit and its corresponding label that is an integer number from zero to nine.
For the evaluation and comparison of the learning results, we used a test set containing 10,000 examples formatted in the same way.

We constructed and trained a multilayer neural network shown in Figure~\ref{fig:neuralnet} for learning the MNIST database.
\begin{figure}[htbp]
  \centering
  \includegraphics[width=\textwidth]{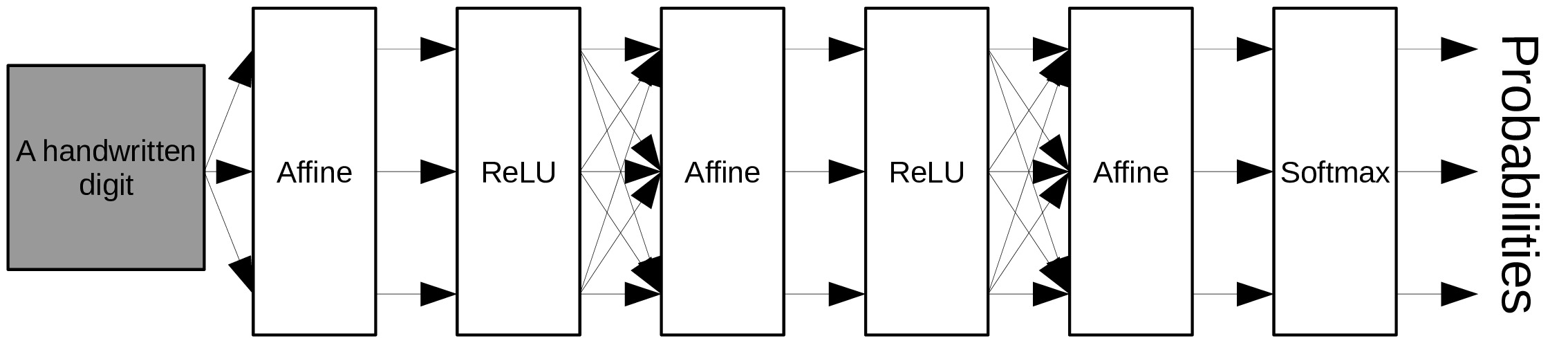}
  \caption{Neural network diagram used to recognizing the MNIST handwritten digits}\label{fig:neuralnet}
\end{figure}
We used three Affine layers with two ReLU (Rectified Linear Unit) activation functions and, for the output, a Softmax activation function.
We used the cross-entropy error function as the objective function for training the neural networks.

An Affine layer $A_{W,b}$ transforms a given vector $x\in\mathbb{R}^n$ into
\begin{align*}
A_{W,b}(x):=Wx+b
\end{align*}
with the parameter $W\in\mathbb{R}^{m\times n}$ and $b\in\mathbb{R}^m$, where $n$ is the number of dimensions of the input vector and $m$ is the number of dimensions of the output vector.
The first Affine layer transforms a $784(=28\times 28)$-dimensional vector, which expresses a given image, into a 300-dimensional vector.
The second Affine layer transforms a 300-dimensional vector into a 100-dimensional vector.
The third Affine layer transforms a 100-dimensional vectors into a 10-dimensional vector, which expresses each probability that the given image is the corresponding number.
We used the number of dimensions described in \cite{lecun1998} for each Affine layer.

The ReLU function transforms each element $x_k\ (k=1,2,\ldots,n)$ of a given vector $x\in\mathbb{R}^n$ into $\max\{x_k, 0\}$.
The Softmax function transforms a given vector $x:=(x_k)_{k=1}^n\in\mathbb{R}^n$ as follows:
\begin{align*}
\mathrm{Softmax}(x):=\frac{1}{\sum_{k=1}^ne^{x_k}}(e^{x_1}, e^{x_2}, \ldots, e^{x_k})^\top.
\end{align*}
We define the cross-entropy error $E:\mathbb{R}^n\to\mathbb{R}$, which is used as the objective function for training neural networks, as follows:
\begin{align*}
E(x):=-\sum_{k=0}^{9}y_k\log(x_k),
\end{align*}
where the vector $x:=(x_k)_{k=0}^{9}\in\mathbb{R}^{10}$ is the output of the current neural network and $y_k\ (k=0,1,2,\ldots,9)$ is one if the label is $k$ and zero otherwise.

In this experiment, we wanted to minimize the cross-entropy error of the training dataset concerning the parameters $W_k, b_k$ for each Affine layer $A_k (k=1,2,3)$.
The number of dimensions of the parameters is $784\times 300+300=235500$ for the first Affine layer, $300\times 100+100=30100$ for the second Affine layer, and $100\times 10+10=1010$ for the third Affine layer.
Hence, the total number of dimensions of the variables for this minimization problem is $235500+30100+1010=266610$.

We ran Algorithm~\ref{alg:ngism} with the Discrete Argmin Line Search described in Algorithm~\ref{alg:dargmin} and compared its behaviors when we used a constant learning rate $\lambda_{n,i}:=0.1$, diminishing learning rate $\lambda_{n,i}:=(0.1\times 20)/n$, and learning rates found by the line search in the step-range $[(0.1\times 20)/(n+100),(0.1\times 20)/n]$ for the number of iterations ($n=1,2,\ldots$).
We set the coefficients for each step-size such that these upper bounds would be equal to each other when the algorithm exits.
To use the proposed algorithm, we have to compute the subgradients of the objective function.
Here, we used approximations of them worked out by the backpropagation algorithm.

We used the computer described in Subsection~2.1 for these experiments.
We wrote the experimental codes in Python~3.6.6 with the NumPy~1.15.4 library.
We divided the datasets into 600 mini-batches, each of which contained 100 examples; in other words, we solved the problem to minimize the sum of 600 objective functions.
We converted and flattened the handwritten digit images into vectors and divided their elements by 255 for regularization.
The parameters for each Affine layer were initialized using a Gaussian distribution of mean zero and variance 0.01.

Figure~\ref{fig:nnfunc} shows the behavior of the values of the objective function for each iteration.
\begin{figure}[htbp]
  \centering
  \includegraphics[width=.45\textwidth]{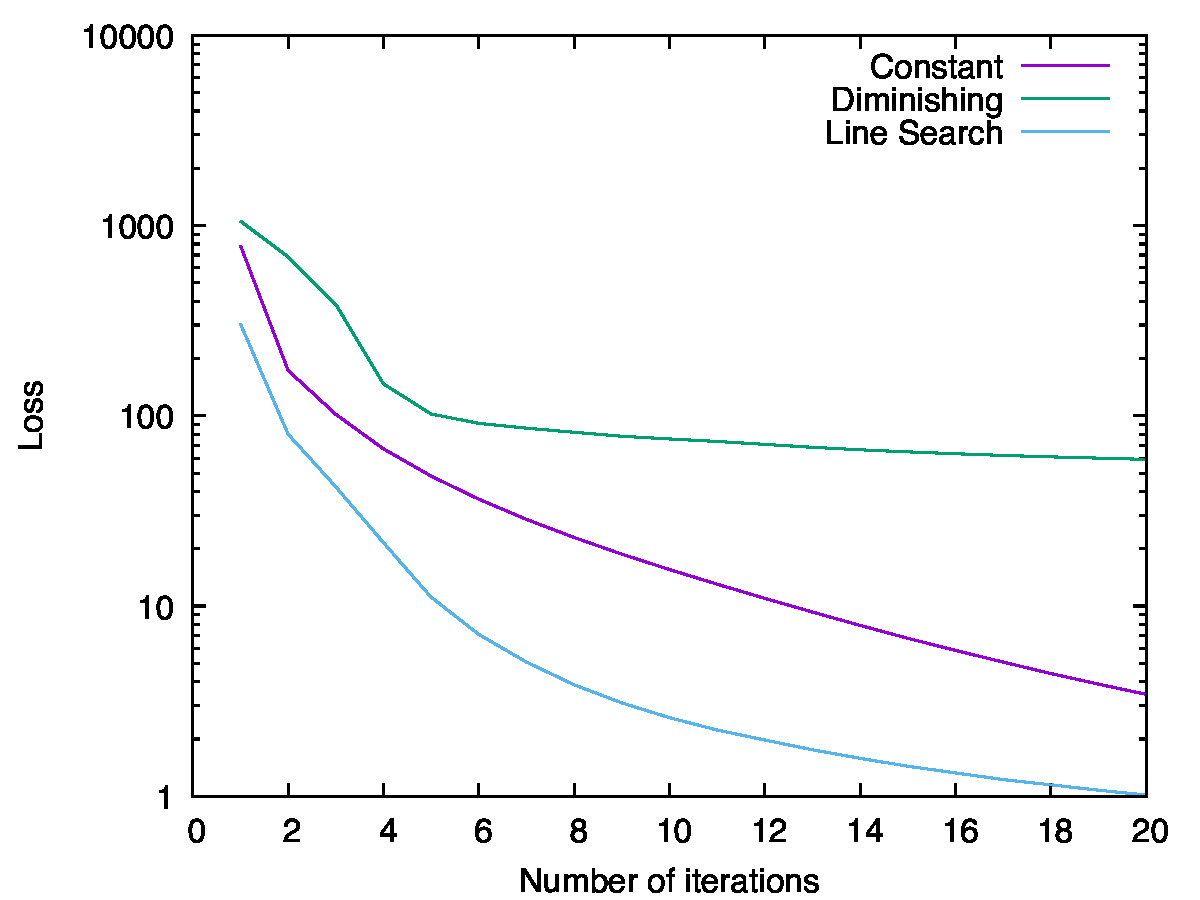}
  \caption{Behavior of the values of the objective function for each iteration}\label{fig:nnfunc}
\end{figure}
The violet line labelled ``Constant'' shows the result of using the constant learning rate, while the green line labelled ``Diminishing'' shows the result of using the diminishing learning rate, and the cyan line labelled ``Line Search'' shows that of using the learning rate computed with the line search.
Overall, we can see that all the results decrease monotonously.
This implies that Algorithm~\ref{alg:ngism} can minimize the objective function with any of the above learning rate settings.
The range of reduction of the result by using the diminishing learning rate is less than others.
One possible reason is that learning rate becomes too small to minimize the objective function sufficiently.
Indeed, from the second to fourth iteration, the result for the diminishing learning rate fell steeply, but this variation became smaller and smaller after the sixth iteration.
In contrast to this result, the results for the constant learning rate and the learning rate computed with the line search minimized the objective function continuously and dramatically.
In particular, we can see that the line search found the most efficient learning rates of these experiments.

Next, let us examine the classification accuracies.
\begin{figure}[htbp]
  \centering
  \subcaptionbox{Result for training data\label{fig:nnscore1}}{\includegraphics[width=.45\textwidth]{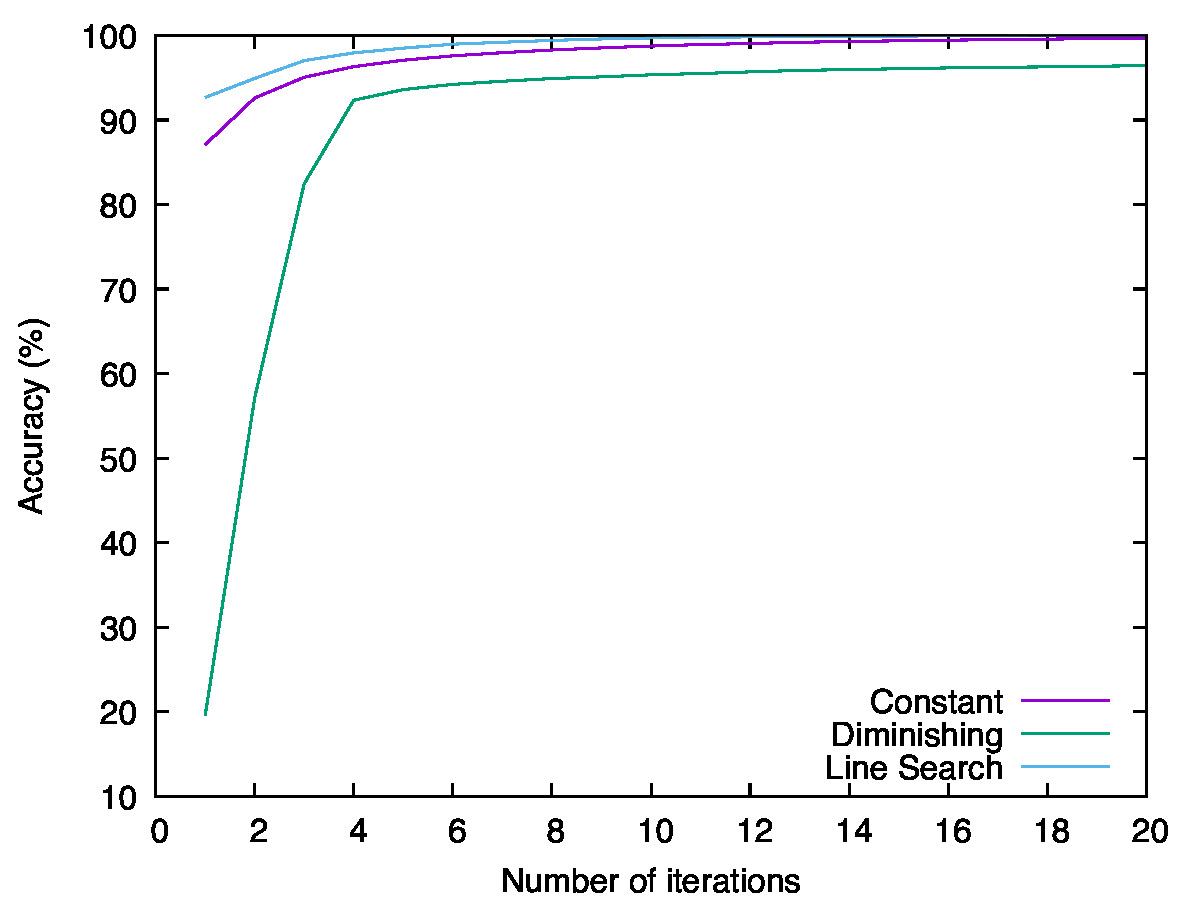}}
  \subcaptionbox{Result for test data\label{fig:nnscore2}}{\includegraphics[width=.45\textwidth]{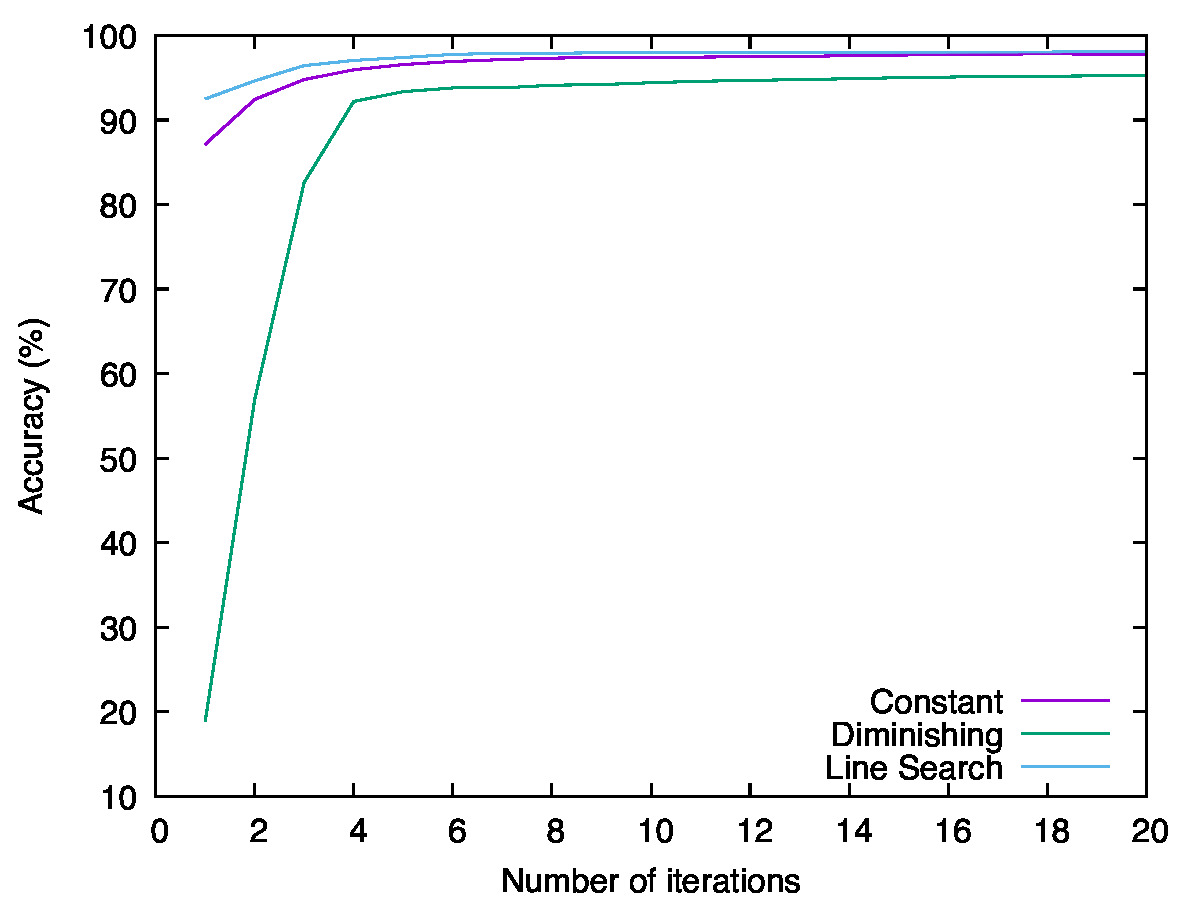}}
  \caption{Behavior of the classification accuracies for training and test data}\label{fig:nnscore}
\end{figure}
Figure~\ref{fig:nnscore} shows the behavior of the classification accuracies for the training and test data.
The left-hand graph (Figure~\ref{fig:nnscore1}) shows the classification accuracies for the training data and the right-hand graph (Figure~\ref{fig:nnscore2}) shows those for the test data.
The legends of these graphs are the same as in Figure~\ref{fig:nnfunc}.
We can see that all the results increased, heading for 100\%.
For both data, the score of ``Line Search'' is higher than others and the score of ``Diminishing'' is the lowest.
This order is the same as what we saw in Figure~2.
Therefore, using the learning rates computed by the line search makes us able to minimize the objective function most and to achieve the best parameters for the neural network to recognize the handwritten digits.

\section{Conclusions}\label{sec:conclu}
We proposed novel incremental and parallel subgradient algorithms with a line search that determines suitable learning rates automatically, algorithmically, and appropriately for learning support vector machines.
We showed that the algorithms converge to optimal solutions of constrained nonsmooth convex optimization problems appearing in the task of learning support vector machines.
Experiments justified the claimed advantages of the proposed algorithms.
We compared them with a machine learning algorithm Pegasos, which is designed to learn with a support vector machine efficiently, in terms of prediction accuracy, value of the objective function, and computational time.
Regarding the parallel subgradient algorithm in particular, the issue of the computational overhead of the line search can be resolved using multi-core computing.
Furthermore, we confirmed that we can apply our incremental subgradient algorithm with the line search to a neural network and they can train it effectively.
Overall, our algorithms are useful for efficiently learning a support vector machine and for training a neural network including deep learning.

\section*{Acknowledgements}
This work was supported by the Japan Society for the Promotion of Science (JSPS KAKENHI Grant Numbers JP17J09220, JP18K11184).
The authors would like to thank the Topic Editor Yoichi Hayashi for giving us a valuable opportunity to submit our research paper to this Research Topic.
We are sincerely grateful to the Topic Editor Guido Bologna and the two anonymous reviewers for helping us improve the original manuscript.
A pre-print version of this paper \cite{arXiv} has been published on the arXiv e-print archive.

\bibliographystyle{abbrv}
\bibliography{biblio}
\end{document}